\documentclass[12pt,twoside]{article}

\newtheorem{prelemmaa}{{\bf LEMMA}}

\newtheorem{prelem}{{\bf THEOREM}}

\newtheorem{preque}{{\bf QUESTION}}

\newtheorem{theorem}{Theorem}

\newtheorem{prelemma}{Lemma}

\newtheorem{preproof}{{\bf Proof:}}

\newenvironment{proof}[1]{\begin{preproof}{\rm
               #1}\hfill{\rule[-0.5mm]{2mm}{2mm}}}{\end{preproof}}

\newtheorem{preproposition}{{PROPOSITION}}

\newtheorem{preremark}{REMARK}

\newtheorem{precorollary}{{COROLLARY}}

\newtheorem{precorn}{{COROLLARY}}

\newtheorem{predefinition}{DEFINITION}

\newtheorem{preexample}{EXAMPLE}

\def\newpic#1{}
\date{}

\title{\bf  On total dominating sets in graphs}
\author{
{\sc Maryam Atapour and Nasrin Soltankhah }\footnote{Corresponding
author: E-mail: soltan@alzahra.ac.ir, \hspace{1.5mm}
soltankhah.n@gmail.com.}
\\
 [6mm]
Department of Mathematics\\ Alzahra University  \\
Vanak Square 19834 \ Tehran, I.R. Iran }

\begin{document}
\maketitle
\begin{abstract}
  A set $S$ of vertices in a graph $G(V,E)$ is called a dominating
 set if every vertex $v\in V$ is either an element of $S$ or is
 adjacent to an element of $S$. A set $S$ of vertices in a graph $G(V,E)$ is called
 a total dominating set if every vertex $v\in V$ is adjacent to an element of $S$. The
domination number of a graph $G$ denoted by $\gamma(G)$ is the
minimum cardinality of a dominating set in $G$. Respectively the
total domination number of a graph $G$ denoted by $\gamma_t(G)$ is
the minimum cardinality of a total dominating set in $G$. An
upper bound for $\gamma_t(G)$ which has been achieved by Cockayne
and et al. in $\cite{coc}$
 is: for any graph $G$ with
no isolated vertex and maximum degree
$\Delta(G)$ and $n$ vertices, $\gamma_t(G)\leq n-\Delta(G)+1$.\\
Here we characterize  bipartite graphs and trees which achieve
this upper bound. Further we present some another upper and lower
bounds for $\gamma_t(G)$. Also,  for circular complete graphs, we
determine the value of $\gamma_t(G)$.
\vspace{3mm}

 \hspace*{-7mm} {\bf Keywords:} {\sf  total
dominating set, total domination number}
\end{abstract}
\section{Introduction}
Let $G(V,E)$ be a graph. For any vertex $x \in V$, we define the
neighborhood  of $x$, denoted by $N(x)$, as the set of all
vertices adjacent to $x$. The closed neighborhood of $x$, denoted
by $N[x]$, is the set $N(x)\cup \{x\}$. For a set of vertices
$S$, we define $N(S)$ as the union of $N(x)$ for all $x\in S$,
and $N[S]=N(S)\cup S$. The degree of a vertex is the size of its
neighborhoods. The maximum degree of a graph $G$ is denoted by
$\Delta(G)$ and the minimum degree is denoted by $\delta(G)$.
Here $n$ will denote the number of vertices of a graph $G$. A set
$S$ of vertices in a graph $G(V,E)$ is called a dominating
 set if every vertex $v\in V$ is either an element of $S$ or is
 adjacent to an element of $S$. A set $S$ of vertices in a graph $G(V,E)$ is called
 a total dominating set if every vertex $v\in V$ is adjacent to an element of $S$. The
domination number of a graph $G$ denoted by $\gamma(G)$ is the
minimum cardinality of a dominating set in $G$. Respectively the
total domination number of a graph $G$ denoted by $\gamma_t(G)$ is
the minimum cardinality of a total dominating set in $G$. clearly
$\gamma(G)\leq \gamma_t(G)$, also it has been proved that
$\gamma_t(G)\leq 2\gamma(G)$.
\\ An upper bound for $\gamma_t(G)$ has been achieved by Cockayne and et
al. in $\cite{coc}$ in the following theorems:
\begin{prelem}\label{t1}
If a graph $G$ has no isolated vertices, then
 $\gamma_t(G)\leq n-\Delta(G)+1$.
\end{prelem}
\begin{prelem}
If $G$ is a connected graph and $\Delta(G)< n-1$, then
$\gamma_t(G)\leq n-\Delta(G)$
\end{prelem}

As a result of the above theorems, if $G$ is a graph with
$\gamma_t(G)=n-\Delta(G)+1$, then $\Delta(G)\geq n -1$. Hence, if
$G$ is a $k$-- regular graph and $\gamma_t(G)=n-k+1$, then $G$ is
$K_n$. As a result of the above theorems, if $G$ is a graph with
$\gamma_t(G)=n-\Delta(G)+1$, then $\Delta(G)\geq n -1$. Hence, if
$G$ is a $k$-- regular graph and $\gamma_t(G)=n-k+1$, then $G$ is
$K_n$. Total domination and  upper bounds on the total domination
number in graphs were intensively investigated, see e. g. (
~\cite{ha}, ~\cite{lw}).
\\
Here we characterize  bipartite graphs and trees which achieve
the upper bound in Theorem~\ref{t1}. Further we present some
another upper and lower bounds for $\gamma_t(G)$. Also, for
circular complete graphs, we determine the value of $\gamma_t(G)$.

It is easy to prove that for $n\geq 3$,
$\gamma_t(C_n)=\gamma_t(P_n)=\frac{n}{2}$  if $n\equiv 0 \pmod 4$
and $\gamma_t(C_n)=\gamma_t(P_n)=\lfloor\frac{n}{2}\rfloor +1 $
otherwise.

for the definitions and notations not defined here we refer
the reader to texts, such as \cite{HAY}.
\section{Other bounds for $\gamma_t(G)$}
In this section we introduce some other upper bounds for
$\gamma_t(G)$.
\begin{theorem}
Let $G$ be a connected graph, then $\gamma_t(G) \geq \lceil
\frac{n}{\Delta(G)}\rceil$.
\end{theorem}
\begin{proof}{
Let $S\subseteq V(G)$ be a total dominating set in $G$. Every
vertex in $S$ dominates at most  $\Delta(G)-1$  vertices of
$V(G)-S$ and dominate at least one of the vertices in $ S$.
Hence, $|S|(\Delta(G)-1)+|S|\geq n$. Since, $S$ is an arbitrary
total dominating set, then $\gamma_t(G) \geq\lceil
\frac{n}{\Delta(G)}\rceil$. }
\end{proof}
If $G=K_n$,  $G=C_{4n}$, or $G=P_{4n}$ then
$\gamma_t(G)=\lceil\frac {n}{\Delta(G)}\rceil$. so the above
bound is sharp.
\
\begin{theorem}
Let $G$ be a graph with diam$(G)=2$ then,  $\gamma_t(G) \leq
\delta(G)+1$.
\end{theorem}
\begin{proof}{
Let $x\in V(G)$ and deg$(x)=\delta(G)$. Since,  diam$(G)=2$, then
$N(x)$ is a dominating set for $G$.\\
 Now $S=N(x)\cup \{x\}$ is a
total dominating set for $G$ and $|S|=\delta(G)+1$. Hence,
$\gamma_t(G) \leq \delta(G)+1$.}
\end{proof}
As we know, $\gamma_t(C_5)= 3$ and also $\delta(C_5)= 2$,
$diam(C_5)= 2$ then $\gamma_t(C_5)= \delta(C_5)+1$. Hence, the
above bound is sharp.
\\
 \begin{theorem}
If $G$ is a connected graph with the girth of length $g(G)\geq 5$
and $\delta(G)\geq 2$, then $\gamma_t(G) \leq n-
\lceil\frac{g(G)}{2}\rceil +1$.
\end{theorem}
\begin{proof}{
Let $G$ be a connected graph with  $g(G)\geq 5$ and let $C$ be a
cycle of length $g(G)$. Remove $C$ from $G$ to form a graph
$G^\prime$. Suppose an arbitrary vertex $v\in V(G^\prime)$, since
$ \delta(G)\geq 2$, then $v$ has at least two neighbors say $x$
and $y$. Let $x,y \in C$. If $d(x,y)\geq 3$, then replacing the
path from $x$ to $y$ on $C$ with the path $x,v,y$ reduces the
girth of $G$, a contradiction. If $d(x,y)\leq 2$, then $x,y,v$
are on either $C_3$ or $C_4$ in $G$, contradicting the hypothesis
that $g(G)\geq 5$.
 Hence, no vertex in
$G^\prime$ has two or more neighbors on $C$. Since $\delta(G)\geq
2$, the graph $G^\prime$ has minimum degree at least
$\delta(G)-1\geq 1$. Then $G^\prime$ has no isolated vertex. Now
let $S^\prime$ be a $\gamma_t $--set for $C$. Then
$S=S^\prime\cup V(G^\prime)$ is  a total dominating set for $G$.
Hence,  $\gamma_t(G) \leq n-\lceil\frac {g(G)}{2}\rceil +1$(note
that $\gamma_t(C) \leq \lfloor\frac {g(G)}{2}\rfloor +1$) .}
\end{proof}
\section {Bipartite graphs with $\gamma_t(G)=n-\Delta(G)+1$}
In this section we charactrize the bipartite graphs achieving the upper bound in the theorem A.
\begin{theorem}
Let $G$ be a bipartite graph with no isolated vertices. Then
$\gamma_t(G)=n-\Delta(G)+1$ if and only if $G$ is a graph in form
of $K_{1,t} \bigcup rK_2$ for $r\geq 0$.
\end{theorem}
\begin{proof}{
If $G$ is $K_{1,t} \cup rK_2 (r\geq 0)$, clearly
$\gamma_t(G)=n-\Delta(G)+1$. Now let $G$ be a bipartite graph with
partitions $A\bigcup B$ and $x\in A$ where ${\rm
deg}(x)=\Delta(G)=t$. We continue our proof in four stages:
 \\
{\textbf Stage 1:} We claim that for every vertex $y\in A-\{ x\}$,
$ N(y)-N(x)\neq \emptyset$. If it is not true, there exists a
vertex in $A-\{ x\}$, say $y$, such that  $N(y)\subseteq N(x)$. So
let $u\in N(y)$, the set
 $S=V-(N(x)\cup \{y\})\bigcup \{ u\}$ is a total dominating set and $|S|=n-\Delta(G)$, a
contradiction.
 So we have $n\geq 2|A|+\Delta(G)-1$.\\
 {\textbf Stage 2:} For every
vertex $y\in A$, let $u_y\in N(y)$. Clearly the set
$S=A\cup(\cup_{y\in A} \{ u_y\} )$ is a total dominating set for
$G$ and $|S|\leq 2|A|$, so $\gamma_t(G)\leq 2|A|$. Now let $y\in
A-\{ x\}$ such that $|N(y)-N(x)|\geq2$. Hence, we have:
 \begin{center}
 $n\geq 2|A|+\Delta(G)$
 \end{center}
  \begin{center}
  $\Rightarrow$
 $\gamma_t(G)+\Delta(G)-1\geq 2|A|+\Delta(G)$
 \end{center}
 \begin{center}
 $\Rightarrow$
 $\gamma_t(G)\geq 2|A|+1$,
 \end{center}
  a contradiction. Hence, for every
 vertex $y\in A-\{ x\}$, $|N(y)-N(x)|=1$.\\
 {\textbf Stage 3:} Let $y\in A-\{x\}$ and
 $N(y)\cap N(x)\neq \emptyset$. Let
 $u\in N(y)\cap N(x)$. Now, $S=(V-N(x)\cup \{ y\})\cup \{u\}$ is a
 total dominating set and $|S|=n-\Delta(G)$.
 So, $\gamma_t(G)\leq n-\Delta(G)$, a contradiction.\\
  {\textbf Stage 4:} Let $y,z\in A-\{ x\}$ and
 $N(y)\cap N(z)\neq \emptyset $. Now
 $S=(V-(\{ z\}\cup N(x)))\cup\{ u\}$, where $u\in N(x)$, is a total
 dominating set and $|S|=n-\Delta(G)$.  So, $\gamma_t(G)\leq n-\Delta(G)$, a contradiction. Hence,
  $G$ is a graph in form of $K_{1,t}\cup rK_2 $.}
 \end{proof}
 \begin{precorn}
 Let $T$ is a Tree. Then $\gamma_t(T)=n-\Delta(T)+1$ if and only
 if $T$ is a star.
 \end{precorn}

\section{Total domination numbers
of circular complete graphs}

If $n$ and $d$ are positive integers with $n\geq2d$, then circular
complete graph $K_{n,d}$ is the graph with vertex set $\{v_0,v_1,
\ldots, v_{n-1}\}$ in which $v_i$ is adjacent to $v_j$ if and only
if $d\leq|i-j|\leq n-d$. In this section we determine the total
domination of circular complete graphs. It is easy to see that
$K_{n,1}$ is the complete graph $K_n$ and $K_{n,2}$ is a circle
on $n$ vertices, therefore we assume that $d\geq3$.

\begin{theorem}\label{TK1}
For $n\geq4d-2$ and $d\geq3$, $\gamma_t(K_{n,d})=2$.
\end{theorem}
\begin{proof}
{Clearly, $\gamma_t(K_{n,d})\geq2$. Let $S=\{v_0, v_{2d-1}\}$. We
will show that $S$ is a total dominating set for $K_{n,d}$. Since
$n\geq4d-2$ and $2d-1\leq 2d$, then $2d-1\leq n-d$. Also $2d-1\geq
d$ since $d\geq3$. Thus $d\leq2d-1\leq n-d$ and $v_0v_{2d-1}\in
E(K_{n,d})$. By definition of $K_{n,d}$, $v_0$ is adjacent to
each of the vertices $v_d,v_{d+1},\ldots, v_{n-d}$.

Now for each $1\leq i\leq d-1$ we have
$$n-d+i-(2d-1)=n-3d+i+1\geq4d-2-3d+i+1\geq d$$
and
$$n-d+i-(2d-1)=n-3d+i+1\leq n-3d+d= n-2d<n-d.$$
Thus $v_{2d-1}$ is adjacent to each of the vertices
$v_{n-d+1},\ldots, v_{n-1}$. On the other hand, for each $1\leq
i\leq d-1$ we have
$$2d-1-i\leq 2d-2\leq3d-2\leq n-d$$
and
$$2d-1-i\geq 2d-1-d+1=d.$$
Hence $v_{2d-1}$ is adjacent to each of the vertices $v_0,
v_1,\ldots, v_{d-1}$ and so $S$ is a total dominating set for
$K_{n,d}$ and  $\gamma_t(K_{n,d})=2$.}
\end{proof}
\begin{theorem}\label{TK2}
For $3d\leq n\leq4d-3$ and $d\geq3$, $\gamma_t(K_{n,d})=3$.
\end{theorem}
\begin{proof}
{Let $S=\{v_0, v_d, v_{2d-1}\}$. We prove that $S$ is a
$\gamma_t(K_{n,d})$- set. Since $d\leq2d-2\leq n-d$, $G[S]$
contains no isolated vertices. Clearly $v_0$ and $v_d$ are
adjacent to each of the vertices $v_d, v_{d+1},\ldots, v_{n-d}$
and  $v_{2d}, v_{2d+1},\ldots, v_{n-d}$ respectively. For $1\leq
i\leq d-1$ we have
$$2d-1-i\leq 2d-1-d+1=d$$
and
$$2d-1-i\leq 2d-2\leq 2d\leq n-d$$
Thus $v_{2d-1}$ is adjacent to each of the vertices $v_1, v_2,
\ldots, v_{d-1}$. Hence $S$ is a total dominating set for
$K_{n,d}$ and so  $\gamma_t(K_{n,d})\leq3$. Now we prove that
there is no total dominating set for $K_{n,d}$ of size 2. Let
$S'=\{u,v\}$ be a $\gamma_t(K_{n,d})$- set. Without loss of
generality, let $u=v_0$ and $v=v_j$. Clearly $d\leq j\leq n-d$.
Since $v_0v_{n-d+1}\notin E(K_{n,d})$, $d\leq n-d+1-j\leq n-d$
and so $1\leq j\leq d+1$. Thus $j=d$ or $j=d+1$. In both cases,
$S'$ is not a total dominating set since $v_2, v_3,\ldots,
v_{d-1}$ are not dominated by $S'$ a contradiction. This
completes the proof.}
\end{proof}

\end{document}